\documentclass[12pt]{amsart}
\usepackage{amsfonts, amssymb, latexsym, epsfig} 

\usepackage{amscd,amssymb}
\usepackage{verbatim}
\usepackage{pstricks}
\usepackage{pst-node}
\usepackage[mathscr]{euscript}

\input xy
\xyoption{all}

\newsymbol\pp 1275
\usepackage{tikz}
\usetikzlibrary{matrix,arrows,backgrounds,shapes.misc,shapes.geometric,patterns,calc,positioning}

\usepackage[colorlinks=true, pdfstartview=FitV, linkcolor=blue, citecolor=blue, urlcolor=blue]{hyperref}

\usepackage{fullpage}
\usepackage{graphicx}
\usepackage{enumerate}
\def\VR{\kern-\arraycolsep\strut\vrule &\kern-\arraycolsep}
\def\vr{\kern-\arraycolsep & \kern-\arraycolsep}

\newtheorem{theorem}{Theorem}

\newtheorem{lemma}{Lemma}
\newtheorem{prop}{Proposition}

\theoremstyle{definition}
\newtheorem{definition}[theorem]{Definition}

\newtheorem{question}[theorem]{Question}

\newtheorem{rmk}{Remark}
\newenvironment{remark}[1][]{\begin{rmk}[#1]\pushQED{\qed}}{\popQED \end{rmk}}

\newtheorem{ex}{Example}
\newenvironment{example}[1][]{\begin{ex}[#1]\pushQED{\qed}}{\popQED \end{ex}}

\newcommand{\overlim}[1]{{\buildrel{#1}\over\longrightarrow\;}}

\newcommand{\Hom}{\operatorname{Hom}}

\newcommand{\Ext}{\operatorname{Ext}}

\newcommand{\rep}{\operatorname{rep}}
\newcommand{\Proj}{\operatorname{Proj}}
\newcommand{\rank}{\operatorname{rank}}

\newcommand{\ZZ}{\mathbb Z}

\newcommand{\C}{K}

\newcommand{\PP}{\mathbb P}

\newcommand{\A}{\mathbb A}

\newcommand{\s}{\mathcal S}

\newcommand{\ima}{\operatorname{Im}}

\newcommand{\Ker}{\operatorname{Ker}}

\newcommand{\rad}{\operatorname{Rad}}
\newcommand{\soc}{\operatorname{Soc}}

\newcommand{\ee}{\operatorname{\mathbf{e}}}

\newcommand{\rr}{\mathbf{r}}

\newcommand{\ekp}{\mathsf{EKP}}
\newcommand{\eip}{\mathsf{EIP}}
\newcommand{\CR}{\mathsf{CR}}
\newcommand{\cjt}{\mathsf{CJT}}
\newcommand{\jt}{\mathsf{Jtype}}
\newcommand{\coh}{\mathsf{Coh}}
\newcommand{\Gr}{\mathsf{Gr}}

\newcommand{\Max}{\mathsf{max}}
\newcommand{\F}{\mathcal{F}}
\newcommand{\V}{\mathcal{V}}
\newcommand{\OO}{\mathcal{O}}
\newcommand{\U}{\mathcal{U}}
\newcommand{\R}{\mathcal{R}}
\newcommand{\x}{\mathbf{t}}

\newcommand{\M}{\operatorname{\mathcal{M}}}

\newcommand{\module}{\operatorname{mod}}
\newcommand{\grmodule}{\operatorname{grmod}}

\newcommand{\rk}{\operatorname{rank}}

\begin{document}
\title{Quiver representations of constant Jordan type and vector bundles}

\author{Andrew Carroll}
\address{University of Missouri-Columbia, Mathematics Department, Columbia, MO, USA}
\email[Andrew Carroll]{carrollat@missouri.edu}

\author{Calin Chindris}
\address{University of Missouri-Columbia, Mathematics Department, Columbia, MO, USA}
\email[Calin Chindris]{chindrisc@missouri.edu}

\author{Zongzhu Lin}
\address{Kansas State University, Manhattan, KS, USA}
\email[Zongzhu Lin]{zlin@math.ksu.edu}

\date{\today}
\bibliographystyle{plain}
\subjclass[2010]{16G20; 14L30}
\keywords{path algebras, modules, constant Jordan type, vector bundles, moduli spaces of thin representations, equal images/kernels properties}

\begin{abstract} Inspired by the work of Benson, Carlson, Friedlander, Pevtsova, and Suslin on modules of constant Jordan type for finite group schemes, we introduce in this paper the class of representations of constant Jordan type for an acyclic quiver $Q$. We do this by first assigning to an arbitrary finite-dimensional representation of $Q$ a sequence of coherent sheaves on moduli spaces of thin representations. Next, we show that our quiver representations of constant Jordan type are precisely those representations for which the corresponding sheaves are locally free. We also construct representations of constant Jordan type with desirable homological properties. Finally, we show that any element of $\ZZ^L$, where $L$ is the Loewy length of the path algebra of $Q$, can be realized as the Jordan type of a virtual representation of $Q$ of relative constant Jordan type.
\end{abstract}

\maketitle
\setcounter{tocdepth}{1}
\tableofcontents

\section{Introduction} Throughout this paper, $K$ is an algebraically closed field of arbitrary characteristic. By a module, we always mean a finite-dimensional left module unless otherwise specified.

A fundamental problem in the representation theory of finite-dimensional algebras is to classify the indecomposable modules. This is, however, a hopeless problem for wild algebras since their representation theory is known to be undecidable. As such, in the presence of wild algebras, one is naturally led to consider special classes of modules. Our goal here is to construct large classes of modules over path algebras of quivers that have distinguished algebraic and geometric characteristics. 

In \cite{CarFriPev1}, Carlson, Friedlander, and Pevtsova have introduced the class of modules of constant Jordan type for finite group schemes. Inspired by their seminal work (see also \cite{CarFriPev2, CarFri, FriPev, BenPev, CarFriSus}), we introduce in this paper the class of modules of constant Jordan type over path algebras of quivers.

Let $Q=(Q_0,Q_1,t,h)$ be an acyclic quiver, $KQ$ its path algebra, and $L$ the Loewy length of $KQ$. To define our $KQ$-modules of constant Jordan type, we fix an effective weight $\sigma_0$ of $Q$ and consider the corresponding moduli space $\M$ of $\sigma_0$-semi-stable thin representations of $Q$. In \cite{Hil2}, Hille showed that $\M$ is a (possibly singular) toric variety. In Section \ref{sec:modules-sheaves-main-defns}, we set up a process that assigns to any $KQ$-module $M$ a sequence $(\F_1(M), \ldots, \F_L(M))$ of coherent sheaves on $\M$. The answer to the question of when these sheaves are locally free leads us to the definition of $KQ$-modules of constant Jordan type.

On the algebraic side, we assign to each point $\alpha \in \A^{Q_1}$ a linear combination $T_{\alpha}$ of arrows of $Q$ that takes into account the toric data defining $\M$. For a given $KQ$-module $M$ and $\alpha \in \A^{Q_1}$, let $\alpha^*(M)$ be the pull-back of $M$ along the algebra homomorphism $K[t]/(t^L) \to KQ$ defined by sending $t+(t^L)$ to $T_{\alpha}$. We then define a $KQ$-module $M$ to be of \emph{constant Jordan type} if the decomposition of $\alpha^*(M)$ into indecomposable $K[t]/(t^L)$-modules does not depend on the choice of $\alpha$ in the $\sigma_0$-semi-stable locus $\V \subseteq \A^{Q_1}$ (see Section \ref{sec:main-defns} for further details). 

Our first result, Theorem \ref{cjt-bundles-thm} in Section \ref{cjt-vectbdle-sec}, simply says that a $KQ$-module $M$ is of constant Jordan type $[L]^{a_L}\dots [2]^{a_2}[1]^{a_1}$ if and only if $\F_i(M)$ is locally free of rank $a_i$ for every $1 \leq i \leq L$. This geometric correspondence is the quiver analog of a result of Benson and Pevtsova on modules of constant Jordan type for elementary abelian $p$-groups (see \cite[Proposition 2.1]{BenPev}). In fact, this geometric result of Benson and Pevtsova (see also \cite[Ch. 7]{Ben2}) has served as the guiding principle behind our definition of a module of constant Jordan type. The key difference in our approach lies in the use of moduli spaces of quiver representations instead of the Friedlander-Pevtsova's $\pi$-point schemes which are not available in the context of representations of quivers.

In Section \ref{cjt-vectbdle-sec}, we also solve the so-called geometric realization problem for tame Kronecker quivers. Specifically, we show in Theorem \ref{grp-tame-Kronecker} that any vector bundle over $\PP^1$ can be realized as $\F_1(M)$ for a module of constant Jordan type over a tame Kronecker quiver. In Section \ref{eip-ekp:sec}, we construct $KQ$-modules with the constant images/kernels properties and those with constant rank, and show they have certain homological features. We finally prove in Section \ref{exact-cjt:sec} that the category $\cjt(Q)$ of $KQ$-modules of constant Jordan type has an exact structure in the sense of Quillen. Moreover, we show that any element of $\ZZ^L$ can be realized as the Jordan type of virtual representations of relative constant Jordan type.

\subsection*{Acknowledgements} The second author would like to thank Tom Nevins for clarifying discussions on descent of coherent sheaves to geometric invariant theory quotients. The second author was supported by NSF grant DMS-1101383.

\section{Modules of constant Jordan type: main definitions and examples}\label{sec:main-defns} Let $Q=(Q_0,Q_1, t,h)$ be a connected acyclic quiver, $KQ$ its path algebra, and $L$ the Loevey length of $KQ$.  

Recall that a finite-dimensional representation $M$ of $Q$ over $K$ is a collection of finite-dimensional $K$-vector spaces $M_x$, $x \in Q_0$, and $K$-linear maps $M_a \in\Hom_K(M_{t(a)},M_{h(a)})$, $a \in Q_1$. Given two representations $M$ and $N$ of $Q$, we define a morphism $\varphi:M \rightarrow N$ to be a collection $(\varphi_x)_{x \in Q_0}$ of $K$-linear maps with $\varphi_x \in \Hom_K(M_x, N_x)$ for each $x \in Q_0$, and such that $\varphi_{h(a)}\circ M_a=N_a\circ \varphi_{t(a)}$ for each $a \in Q_1$. We denote by $\Hom_Q(M,N)$ the $K$-vector space of all morphisms from $M$ to $N$. 

The category of finite-dimensional representations of $Q$ is equivalent to the category $\module(KQ)$ of $KQ$-modules. In fact, we use interchangeably the vocabulary of $KQ$-modules and that of representations of $Q$. For each vertex $x \in Q_0$, we denote the simple (one-dimensional) $KQ$-module supported at vertex $x$ by $S(x)$ and its projective cover by $P(x)$. For background on quivers and their representations, we refer the reader to \cite{AS-SI-SK}.

In what follows, we first explain how to associate to a $KQ$-module a sequence of coherent sheaves on moduli spaces of thin representations of $Q$ by adapting the strategy from \cite[Section 2]{BenPev} to our set-up (see also \cite{FriPev} and  \cite{Ben2}). We then identify those modules for which the corresponding sequence of sheaves consists of vector bundles.

\subsection{Modules and coherent sheaves}\label{sec:modules-sheaves-main-defns} The torus $(K^*)^{Q_0}$ acts on $\A^{Q_1}$ by $$t\cdot \alpha=(t_{ha}t_{ta}^{-1}\alpha_a)_{a \in Q_1}, \forall t=(t_i)_{i \in Q_0} \in (K^*)^{Q_0}, \alpha=(\alpha_a)_{a\in Q_1} \in \A^{Q_1}.$$ 

Let $\{\ee_i \mid i \in Q_0\}$ be the standard $\ZZ$-basis of the lattice $\ZZ^{Q_0}$ and let $H=\{\sigma=(\sigma_i)_{i \in Q_0} \in \ZZ^{Q_0} \mid \sum_{i \in Q_0} \sigma_i=0\}$. Note that $H$ is a lattice of rank $|Q_0|-1$. The above torus action induces an $H$-grading on the polynomial ring
$$
\s=K[Y_a: a \in Q_1]=\bigoplus_{\sigma \in H} \s_{\sigma}.
$$
For each $\sigma \in H$, $\s_{\sigma}=\{f \in \s \mid t \cdot f=\left(\prod_{i \in Q_0} t_i^{\sigma_i}\right)f, \forall t=(t_i)_{i \in Q_0} \in (K^*)^{Q_0}\}$ is the space of semi-invariants of weight $\sigma$. Note that $Y_a \in \s_{\ee_{ta}-\ee_{ha}}$ for every arrow $a \in Q_1$. 

We fix a non-zero integral weight $\sigma_0 \in H$ for which the corresponding $\sigma_0$-semi-stable locus $(\A^{Q_1})^{ss}_{\sigma_0}$ is non-empty. (For an explicit combinatorial description of $(\A^{Q_1})^{ss}_{\sigma_0}$, see \cite{Hil2} or \cite{AltHil}.) Let $\M:=\Proj(\bigoplus_{n \geq 0} \s_{n \sigma_0})$ be the GIT-quotient of $\V:=(\A^{Q_1})^{ss}_{\sigma_0}$ by $T:=(K^*)^{Q_0}/K^*$ and let $\pi: \V \to \M$ be the good quotient morphism. Recall that $\M$ is a (possibly singular) toric variety, called a toric quiver variety in \cite{Hil2} whenever $\sigma_0$ is chosen to be generic.

Let $I$ be the finite set of lattice points of the polytope of flows with given input $\sigma_0$ (see \cite{AltHil}). Specifically, we have that 
$$I=\{\rr=(r_a)_{a \in Q_1} \in \ZZ_{\geq 0}^{Q_1} \mid \sum_{\buildrel {a \in Q_1} \over {ta=i}} r_a-\sum_{\buildrel {a \in Q_1} \over {ha=i}} r_a=\sigma_0(i), \forall i \in Q_0 \}.$$ 
If $Y_{\rr}:=\prod_{a \in Q_1} Y_a^{r_a}$, $\rr \in I$, then $\{Y_{\rr} \mid \rr \in I\}$ is a $K$-basis for $\s_{\sigma_0}$ (see \cite{CC6}). Note also that $\bigoplus_{n \geq 0} \s_{n \sigma_0}$ is generated by $\s_{\sigma_0}$ as a $K$-algebra (see for example \cite[Ex. 10.13]{MilStu}). In particular, we get that $\V=\A^{Q_1} \setminus \mathbb V(Y_{\rr} \mid \rr \in I)$. 

For each $\alpha=(\alpha_a)_{a \in Q_1} \in \A^{Q_1}$ and $\rr=(r_a)_{a \in Q_1} \in I$, we define:
\begin{itemize}
\item $\alpha_{\rr}=Y_{\rr}(\alpha)=\prod_{a \in Q_1} \alpha_a^{r_a} \in K$;
\item $X_{\rr}=\sum_{a \in Q_1}r_a X_a \in KQ$, where $X_a$ denotes the arrow $a$ in $KQ$;
\item $T_{\alpha}=\sum_{\rr \in I}\alpha_{\rr} X_{\rr} \in KQ$.
\end{itemize}

\begin{example} \label{ex:basicexamples}
\begin{enumerate}
  
  

\item Consider the generalized Kronecker quiver with $n \geq 2$ arrows: 
$$\mathcal{K}_n:~
\vcenter
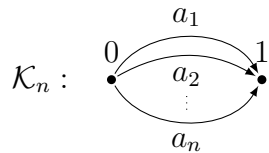
{\hbox{  
\begin{tikzpicture}[point/.style={shape=circle, fill=black, scale=.3pt,outer sep=3pt},>=latex]
   \node[point,label={above:$0$}] (1) at (0,0) {};
   \node[point,label={above:$1$}] (2) at (2,0) {};
  
   \draw[dotted] (1,-.4)--(1,-.2);
  
   \path[->]
   (1) edge [bend left=60] node[midway, above] {$a_1$} (2)
   (1) edge [bend left=30] node[midway, below] {$a_2$} (2)
   (1) edge [bend right=60] node[midway, below] {$a_n$} (2);
\end{tikzpicture} 
}}
$$
and let $\sigma_0(1)=1$, $\sigma_0(2)=-1$. Then, $I$ is the set of the standard basis vectors of $\ZZ^n$ and, for any $\alpha=(\alpha_1, \ldots, \alpha_n) \in \A^n$, $T_{\alpha}=\sum_{i=1}^n \alpha_i X_{a_i}$. Moreover, $\V=\A^n\setminus \{0\}$ and $\M=\PP^{n-1}$. 

More generally, toric degenerations of partial flag varieties (in type $\mathbb{A}$) can be also realized as moduli spaces $\M$ of thin representations of flag quivers (see \cite{BatCioKimvanStr} and \cite{AltvanS}).  

\item Consider the quiver:
$$
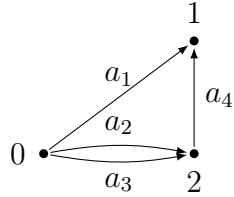
\vcenter{\hbox{  
\begin{tikzpicture}[point/.style={shape=circle, fill=black, scale=.3pt,outer sep=3pt},>=latex]
   \node[point,label={left:$0$}] (1) at (0,0) {};
   \node[point,label={above:$1$}] (2) at (2,1.5) {};
   \node[point,label={below:$2$}] (3) at (2,0) {};
   
  \path[->]
   (1) edge  node[midway, above] {$a_1$} (2)
   (1) edge [bend left=10] node[midway, above] {$a_2$} (3)
   (1) edge [bend right=10] node[midway, below] {$a_3$} (3)
   (3) edge node[midway, right] {$a_4$} (2);
\end{tikzpicture} 
}} 
$$
and let $\sigma_0(0)=2$, $\sigma_0(1)=\sigma_0(2)=-1$. Then, 
$$I=\{(1,1,0,0), (1,0,1,0), (0,2,0,1), (0,1,1,1), (0,0,2,1)\},$$
$\V=\A^4 \setminus \mathbb{V}(Y_{a_1}Y_{a_2}, Y_{a_1}Y_{a_3}, Y_{a_2}^2Y_{a_4}, Y_{a_2}Y_{a_3}Y_{a_4}, Y_{a_3}^2Y_{a_4})$, and $\M$ is the Hirzebruch surface $\mathbb F_1$ (see \cite{CraSmi}). 

More generally, Craw and Smith found in \cite{CraSmi} an effective way of realizing an arbitrary toric variety as a fine moduli space of thin representations of bound quivers. 
\end{enumerate}
\end{example}

We are now ready to describe the process that assigns to a $KQ$-module $M$ a sequence of coherent sheaves $(\F_1(M), \ldots, \F_L(M))$ on $\M$. Let $\R$ be the composition of functors:
$$
\R: \grmodule_H(\s) \overlim{\tilde{}} \coh^T(\A^{Q_1})\overlim{i^{\star}} \coh^T(\V) \overlim{\pi_{\star}^T} \coh(\M),
$$
where $\grmodule_H(\s)$ is the category of finitely-generated $H$-graded $\s$-modules, $\tilde{}$ is the standard tilde functor, $i:\V\to \A^{Q_1}$ is the open immerison, and $\pi:\V \to \M$ is the good quotient morphism. Explicitly,  if $M \in \grmodule_H(\s)$ and $f \in \s_{l\sigma_0}$ with $l \geq 1$ then
$$
\R(M)(\U_f)=\left\{\left.{m\over f^n} \right\rvert n \in \ZZ_{\geq 0}, m \in M \text{~is homogeneous with} \deg_H(m)=n(l\sigma_0)\right\},
$$
where $\U_f=\pi(\V_f) \subseteq \M$ and $\V_f\subseteq \V$ are the principal open subsets defined by $f$. Note that $\R$ is an additive exact functor. 

For any integer $j \in \ZZ$, we define $\s(j)$ to be the $H$-graded polynomial algebra $\s$ whose $\sigma$-degree part is $\s_{\sigma+j\sigma_0}$, i.e. $\s(j)=\oplus_{\sigma \in H} \s_{\sigma+j\sigma_0}$. Then, $\R(\s)=\OO$ and, more generally, $\R(\s(j))=\mathcal O(j), \forall j \in \ZZ$, where $\OO$ is the structure sheaf on $\M$. Moreover, if $\F$ is a sheaf on $\M$, $\F(j)$ denotes the twist $\F \otimes_{\mathcal O} \mathcal O(j)$.

For a $KQ$-module $M$, we denote by $\widetilde{M}$ the trivial vector bundle $M \otimes_K \mathcal O$ over $\M$ of rank $\dim_K M$; of course, one has that $\widetilde{M}(j)=M\otimes_K \mathcal O(j)$ with underlying $H$-graded $\s$-module $M\otimes_K \s(j)$, i.e. $\R(M\otimes_K \s(j))=\widetilde{M}(j)$ for all $j \in \ZZ$.

We have the map of vector bundles $\theta_M: \widetilde{M} \to \widetilde{M}(1)$ defined by the formula 
$$\theta_M(m\otimes f)=\sum_{\rr \in I} X_{\rr}m\otimes Y_{\rr} f.$$
When the module $M$ is understood from the context, we simply write $\theta$ for $\theta_M$. We also use the same $\theta$ to denote the twist $\theta(j): \widetilde{M}(j) \to \widetilde{M}(j+1)$ for any $j \in \ZZ$. With this convention, we have $\theta^L=0$. Moreover, whenever we work with $\Ker \theta^i, i \geq 0$, it is understood that $\theta^i$ is the composition of the maps $\widetilde{M}\xrightarrow{\theta} \ldots \xrightarrow{\theta} \widetilde{M}(i)$; in particular, $\Ker \theta^i$ is a coherent subsheaf of $\widetilde{M}$. Similarly, whenever we work with $\ima \theta^i, i \geq 0$, it is understood that $\theta^i$ is the composition of the maps $\widetilde{M}(-i) \xrightarrow{\theta} \ldots \xrightarrow{\theta} \widetilde{M}$; in particular, $\ima \theta^i$ is a coherent subsheaf of $\widetilde{M}$.

For each $1 \leq i \leq L$, the fibers of the sheaf $\ima \theta^i$ can be easily described. Specifically, let $\overline{\alpha} \in \M$ be a point and choose $\alpha \in \V$ so that $\pi(\alpha)=\overline{\alpha}$. Assume $Y_{\rr}(\alpha) \neq 0$ for some $\rr \in I$ and denote by $\U_{\rr} \subseteq \M$ the image of $\{y \in \V \mid Y_{\rr}(y) \neq 0\}$ under $\pi$. This is an open affine neighborhood of $\overline{\alpha}$ in $\M$ with $\mathcal O(\U_{\rr})=(\s[{1 \over Y_{\rr}}])^T=\{{f \over Y^m_{\rr}} \mid m \in \ZZ_{\geq 0}, f\in \s_{m\sigma_0}\}$. Then, for any $j \in \ZZ$, the fiber of $\widetilde{M}(j)$ at $\overline{\alpha}$ is 
$$
\widetilde{M}(j)|_{\overline{\alpha}}=M\otimes_K \OO(j)(\U_{\rr})\otimes_{\OO(\U_{\rr})}K,
$$
where $K$ is regarded as an $\OO(\U_{\rr})$-module via evaluation at $\alpha$. We identify $\OO(j)(\U_{\rr})=\{Y_{\rr}^j\phi \mid \phi \in \OO(\U_{\rr})\}$ with $\OO(\U_{\rr})$ as $\OO(\U_{\rr})$-modules, and then identify the fiber $\widetilde{M}(j)|_{\overline{\alpha}}$ with the $K$-vector space $M$. After making these identifications, the map $\theta$ at the level of fibers,
$\theta|{\overline{\alpha}}:M \to M$, maps 
$$
m\otimes 1 \otimes 1\in M\otimes_K \OO(\U_{\rr})\otimes_{\OO(\U_{\rr})}K\simeq M$$
to
$$
\begin{aligned}
\sum_{\x \in I} X_{\x} m \otimes {Y_{\x}\over Y_{\rr}} \otimes 1&=\sum_{\x \in I} X_{\x}m\otimes 1 \otimes \alpha_{\rr}^{-1}\alpha_{\x}\\
&=\sum_{\x \in I}  \alpha_{\rr}^{-1} \alpha_{\x} X_{\x}m\otimes 1 \otimes 1\\
&=\alpha_{\rr}^{-1}T_{\alpha}m \otimes 1 \otimes 1,
\end{aligned}
$$ 
i.e., $\theta|{\overline\alpha}$ is, up to a non-zero scalar, the linear operator on $M$ that  maps $m \in M$ to $T_{\alpha}m$.

\begin{definition} Given a module $M \in \module(KQ)$, we define for each $\alpha \in \A^{Q_1}$ the linear operator $\alpha_M:M \to M$ by
$$
\begin{aligned}
\alpha_M(m)=& T_{\alpha}m=\\
=&\sum_{a \in Q_1}\left(\sum_{\rr \in I} r_aY_{\rr}(\alpha) \right) X_a m, \forall m \in M.
\end{aligned}
$$
\end{definition}

\begin{remark} Note that $\alpha_M$ is a nilpotent operator with $\alpha_M^L=0$. Hence, $\alpha_M$ is uniquely determined by its Jordan canonical form.
\end{remark}

The simple computation above proves the following key result:

\begin{prop}\label{fiber-prop} Keeping the same notation as above, the fiber of $\ima \theta^i$ at $\overline{\alpha} \in \M$ is isomorphic to $\ima \alpha_M^i$.
\end{prop}

We obtain more precise information about $\alpha_M$ by considering the ``kernel filtration'' and the ``image filtration'' of $\widetilde{M}$:
$$
0\subset \Ker \theta \subset \ldots \subset \Ker \theta^{L-1} \subset \widetilde{M},
$$
$$
0=\ima \theta^L \subset \ima \theta^{L-1} \subset \ldots \subset \ima \theta \subset \ima \theta^0=\widetilde{M}.
$$
The quotients of the standard refinement of the kernel filtration by the image filtration are the sheaves that will be of interest to us. For simplicity, set $\mathcal K_j=\Ker \theta^j$ and $\mathcal I_j=\ima \theta^{L-j}$. We now refine each step $\mathcal K_j \subseteq \mathcal K_{j+1}$ by
$$
\mathcal K_j\subseteq (\mathcal K_{j+1} \cap \mathcal I_1)+\mathcal K_j \subseteq \ldots \subseteq (\mathcal K_{j+1} \cap \mathcal I_i)+\mathcal K_j \subseteq  (\mathcal K_{j+1} \cap \mathcal I_{i+1})+\mathcal K_j \subseteq \ldots \subseteq \mathcal K_{j+1}.$$
Recall that for three sheaves $A,B,C$ with $B \subseteq A$, one has:
$$
{A+C \over B+C} \simeq {A\over B+(A \cap C)}.
$$
This isomorphism is a consequence of the second isomorphism theorem and the modular law (see for example \cite[Section 7.4]{Ben2}). So, the factors of the refined kernel filtration are of the form:
$$
{(\mathcal K_{j+1} \cap \mathcal I_{l+1})+\mathcal K_j  \over (\mathcal K_{j+1} \cap \mathcal I_l)+\mathcal K_j} \simeq {\mathcal K_{j+1} \cap \mathcal I_{l+1}   \over (\mathcal K_{j+1} \cap \mathcal I_l)+(\mathcal K_j \cap \mathcal I_{l+1})}={\Ker \theta^{j+1} \cap \ima \theta^{i-j-1}   \over (\Ker \theta^{j+1} \cap \ima \theta^{i-j})+(\Ker \theta^j \cap \ima \theta^{i-j-1})} ,
$$
where $i=L-l+j$. Note that when $j>l$, the quotient above becomes $\mathcal I_{l+1}/(\mathcal I_l+\mathcal I_{l+1})=0$.

We are thus lead to define for $0 \leq j<i \leq L$ the functor $$\F_{i,j}:\module(KQ) \to \coh(\M)$$ by
$$
\F_{i,j}(M)={\Ker \theta^{j+1} \cap \ima \theta^{i-j-1} \over (\Ker \theta^{j+1} \cap \ima \theta^{i-j})+(\Ker \theta^j \cap \ima \theta^{i-j-1})}.$$ 

For $1 \leq i \leq L$, we define $$\F_i(M)=\mathcal F_{i,0}(M)={\Ker \theta \cap \ima \theta^{i-1} \over \Ker \theta \cap \ima \theta^i}.$$

We summarize the discussion above in the following lemma: 

\begin{lemma}\label{basic-prop-sheaves-lemma} Keeping the same notation as above, the following statements hold. 
\begin{enumerate} 
\item $\widetilde{M}$ has a filtration whose quotients are isomorphic to $\F_{i,j}(M), 0 \leq j<i \leq L$.
\item The subsheaf $\Ker \theta^i \subseteq \widetilde{M}$ has a filtration whose quotients are isomorphic to $\mathcal F_{j,l}$ with $l \leq i$.
\item The subsheaf $\ima \theta^i \subseteq \widetilde{M}$ has a filtration whose quotients are isomorphic to $\mathcal F_{j,l}$ with $j-l>i$.
\item There is a natural isomorphism $\F_{i,j}(M) \simeq \F_i(M)(j), \forall 0 \leq j<i$.
\end{enumerate}
\end{lemma}

\begin{remark} Even though the arguments in the proof of \cite[Lemmas 2.2 and 2.3]{BenPev} (see also \cite[Lemma 7.4.8]{Ben2}) carry over to our quiver set-up, we include the short proof below for completeness.
\end{remark}

\begin{proof} The first three claims follow immediately from the filtrations described above. To establish the isomorphism in $(4)$, we work first in the category $\grmodule_H(S)$. In fact, any endomorphism in $\module(S)$, such as $\theta$, induces an isomorphism of $S$-modules:
$$
{\Ker \theta^{j+1} \cap \ima \theta^{i-j-1} \over (\Ker \theta^{j+1} \cap \ima \theta^{i-j})+(\Ker \theta^j \cap \ima \theta^{i-j-1})} \simeq {\Ker \theta^j \cap \ima \theta^{i-j} \over (\Ker \theta^j \cap \ima \theta^{i-j+1})+(\Ker \theta^{j-1} \cap \ima \theta^{i-j})}.
$$
Since our homomorphism $\theta: \s\to \s(1)$ is $H$-graded, the isomorphism above becomes an isomorphism of $H$-graded $S$-modules after shifting the module on the right by $1$. Applying the exact functor $\R$, we get a natural isomorphism:
$$
\F_{i,j}(M)\simeq \F_{i,j-1}(M)(1).
$$ 
It now follows by induction that $\F_{i,j}(M) \simeq \F_{i,0}(M)(j)=\F_i(M)(j)$.
\end{proof}

\begin{example} For a semi-simple module $M$ of dimension $n$, $\F_1(M)\simeq \OO^n$ and $\F_i(M)=0$ for $i \geq 2$; in particular, $\F_1(S(x)) \simeq \OO$ for every vertex $x \in Q_0$. 
\end{example}

\begin{remark}
We point out that our functors $\F_i$, $1 \leq i \leq L$, are not well-behaved with respect to syzygies. Specifically, in Example \ref{ex:basicexamples}{(1)}, $\F_1(S(1))\simeq\OO$ and $\F_1(\Omega(S(1)))=F_1(S(2)^{n})\simeq\OO^{n}$. In particular, this shows that Theorem 3.2 in \cite{BenPev}, with $p$ replaced by $L$, does not hold in our quiver set-up.
\end{remark}

\subsection{Modules of constant Jordan type and vector bundles}\label{cjt-vectbdle-sec} Motivated by the discussion above, our goal in this subsection is to understand those modules $M \in \module(KQ)$ for which the corresponding sheaves $\F_1(M), \ldots, \F_L(M)$ are vector bundles over $\M$. 

\begin{definition} A module $M \in \module(KQ)$ is said to have \emph{constant Jordan type} $[L]^{a_L}\dots [2]^{a_2}[1]^{a_1}$ if, for any $\alpha \in \V$, the Jordan canonical form of $\alpha_M$ has $a_L$ Jordan blocks of length $L$, $\ldots$, $a_1$ Jordan blocks of length $1$.
\end{definition}

If $M \in \module(KQ)$ has constant Jordan type $[L]^{a_L}\dots [2]^{a_2}[1]^{a_1}$, we define the Jordan type of $M$ to be:
$$
\jt(M)=[L]^{a_L}\dots [2]^{a_2}[1]^{a_1}.
$$

\begin{example} (1) Any semi-simple module of dimension $n$ over a path algebra is of constant Jordan type $([1]^{\dim_K M})$.\\

(2) For the generalized Kronecker quiver $\mathcal{K}_n$ and weight $\sigma_0$ from Example \ref{ex:basicexamples}{(1)}, our definition of modules of constant Jordan type coincides with Worch's definition from \cite{Wor}. 

When $n=2$, it is immediate to see that the preprojective and preinjective indecomposable modules are the only indecomposable modules of constant Jordan type.\\

(3) For the quiver $Q$ and weight $\sigma_0$ from Example  \ref{ex:basicexamples}{(2)}, it is easy to see that $P(2)$ does not have constant Jordan type. This is in contrast to the situation for elementary abelian $p$-groups where any projective module is of constant Jordan type.
\end{example}

\begin{remark} For a module $M \in \module(KQ)$ and $\alpha \in \A^{Q_1}$, let us denote by $\alpha^*(M)$ the pull-back of $M$ along the algebra homomorphism $K[t]/(t^L)\to KQ$ defined by sending $t+(t^L)$ to $T_{\alpha}$. Then, $M$ has constant Jordan type if and only if the decomposition of $\alpha^*(M)$ into indecomposable $K[t]/(t^l)$-modules does not depend on the choice of $\alpha$ in $\V$.
\end{remark}

\begin{remark}
We point out that from the perspective of our newly defined modules of constant Jordan type, the GIT-quotient $\M$ plays for us the role of the $\pi$-point scheme $\Pi(\mathcal G)$ for finite group schemes $\mathcal G$.
\end{remark}

We are now ready to prove our first result which is the quiver analog of \cite[Proposition 2.1]{BenPev}. Although the proof strategy is the same as that in ibid., we nevertheless provide the proof below for completeness. 

\begin{theorem} \label{cjt-bundles-thm}
A module $M\in \module(KQ)$ has constant Jordan type $[L]^{a_L}\ldots [1]^{a_1}$ if and only if the sheaf $\F_i(M)$ is locally free of rank $a_i$ for all $1 \leq i \leq L$.  
\end{theorem}

\begin{proof}[Proof of Theorem \ref{cjt-bundles-thm}] First let us assume that $\F_i(M)$ is locally free of rank $a_i$ for every $1 \leq i \leq L$. We know from Lemma \ref{basic-prop-sheaves-lemma} that, for every $1 \leq i \leq L$, $\ima \theta^i$ has a filtration whose quotients are isomorphic to $\F_j(M)(l)$ with $j-l>i$. Since these quotients are locally free sheaves, $\ima \theta^i$ is also locally free (see for example \cite[Proposition 5.2.7]{Ben2}). Next, we compute the ranks of the vector bundles $\ima \theta^i$, $1 \leq i \leq L$. Recall that $\mathcal K_i$ denotes $\Ker \theta^i$ and $\mathcal I_i$ denotes $\ima \theta^{L-i}$. Now, for each $0 \leq i \leq L-1$, the factors of the refinement of the step $\mathcal I_i \subseteq \mathcal I_{i+1}$ by the kernel filtration are of the form:
$$
{\mathcal I_{i+1} \cap \mathcal K_{l+1} \over (\mathcal I_{i+1} \cap \mathcal K_l)+(\mathcal I_i \cap \mathcal K_{l+1})}=\F_{L+l-i,l}(M)\simeq \F_{L+l-i}(M)(l), 0 \leq l \leq i, 
$$
and consequently we obtain: 
$$
\rk \ima \theta^{i-1}-\rk \ima \theta^i=\sum_{j=0}^{L-i} a_{i+j}, \forall 1 \leq i \leq L.
$$
From these relations and Proposition \ref{fiber-prop} , we get that for any $\alpha \in \V$:

 $$\rk \alpha_M^i=\rk \ima \theta^i=\sum_{j=i}^L a_j(j-i), \forall 1 \leq i \leq L,$$ and hence the number of Jordan blocks  of size $i \times i$ in the Jordan canonical form of $\alpha_M$ is:
$$
\rk \alpha_M^{i+1}+\alpha_M^{i-1}-2\rk \alpha_M^i=a_i.
$$ 
Therefore, $M$ has constant Jordan type $[L]^{a_L}\ldots [1]^{a_1}$.

Conversely, let us assume that $M$ has constant Jordan type $[L]^{a_L}\ldots [1]^{a_1}$. Let $\overline{\alpha}=\pi(\alpha)\in \M$, with $\alpha \in \V$, be an arbitrary point of $\M$ and $1 \leq i \leq L$. We will show that $\dim_K\F_i(M)_{\overline{\alpha}}\otimes_{\OO_{\overline{\alpha}}}\kappa(\overline{\alpha})=a_i$. 

By definition, we have the short exact sequence of coherent sheaves: 
$$
0 \to \Ker \theta \cap \ima \theta^i \to \Ker \theta \cap \ima \theta^{i-1} \to \F_i(M) \to 0,
$$
which gives rise to:
\begin{equation}\label{eqn-1}
(\Ker \theta \cap \ima \theta^i)_{\overline{\alpha}}\otimes_{\OO_{\overline{\alpha}}}\kappa(\overline{\alpha}) \to (\Ker \theta \cap \ima \theta^{i-1})_{\overline{\alpha}}\otimes_{\OO_{\overline{\alpha}}}\kappa(\overline{\alpha}) \to \F_i(M)_{\overline{\alpha}}\otimes_{\OO_{\overline{\alpha}}}\kappa(\overline{\alpha}) \to 0.
\end{equation}
Note that each $\Ker \theta \cap \ima \theta^i$ is a locally free sheaf of rank $\sum_{j=i+1}^L a_j$ whose fiber at $\overline{\alpha}$ is $\Ker \alpha_M \cap \ima \alpha_M^i$. Indeed, as $\Ker \theta \cap \ima \theta^i=\Ker (\theta: \ima \theta^i \to \ima \theta^{i+1})$, we have the short exact sequence of coherent sheaves:
$$
0 \to \Ker \theta \cap \ima \theta^i \to \ima \theta^i \to \ima \theta^{i+1} \to 0.
$$
Since $\ima \theta^i$ and $\ima \theta^{i+1}$ are locally free sheaves by Proposition \ref{fiber-prop}, we know that $\Ker \theta \cap \ima \theta^i$ is a locally free as well (see for example \cite[Lemma 5.2.4]{Ben2}); its rank can be easily seen to be $\sum_{j=i+1}^L a_j$. So, we have the short exact sequence of vector spaces:
$$
0 \to (\Ker \theta \cap \ima \theta^i)_{\overline{\alpha}} \otimes_{\OO_{\overline{\alpha}}}\kappa(\overline{\alpha}) \to \ima \theta^i_{\overline{\alpha}}\otimes_{\OO_{\overline{\alpha}}}\kappa(\overline{\alpha}) \to \ima \theta^{i+1}_{\overline{\alpha}}\otimes_{\OO_{\overline{\alpha}}} \kappa(\overline{\alpha}) \to 0.
$$
Using Proposition \ref{fiber-prop} again, it is now clear that the fiber of $\Ker \theta \cap \ima \theta^i$ at $\overline{\alpha}$ is $\Ker(\alpha_M: \ima \alpha_M^i \to \ima \alpha_M^{i+1})=\Ker \alpha_M \cap \ima \alpha_M^i$. This proves that the the leftmost linear map of $(1)$ is injective as $\Ker \alpha_M \cap \ima \alpha_M^i \subseteq \Ker \alpha_M \cap \ima \alpha_M^{i-1}$. Consequently, we get that $\dim_K\F_i(M)_{\overline{\alpha}}\otimes_{\OO_{\overline{\alpha}}}\kappa(\overline{\alpha})=a_i$. 
\end{proof}

At this point, it is natural to ask which vector bundles over $\M$ can be realized as $\F_i(M)$ for $M\in \module(\C Q)$ of constant Jordan type and $1 \leq i \leq L$. We refer to this problem as the \emph{geometric realization problem} for modules of constant Jordan type. In the context of finite abelian $p$-groups this problem has been answered by Benson and Pevtsova in \cite{BenPev}. Their proof is highly non-trivial and relies heavily on the specifics of the set-up; in particular, it is not clear how to adapt their proof to our non-commutative, quiver set-up. 

Next, we look into the geometric realization problem for tame Kronecker quivers. 

\begin{theorem}\label{grp-tame-Kronecker}
Let $\mathcal{K}_2$ be the Kronecker quiver 
$\xymatrix{1\ar@/^/[r]^{a_1} \ar@/_/[r]_{a_2} &2}$ and let us fix the weight $\sigma_0=(1,-1)$. Then, for any integer $n \in \ZZ_{\geq 0}$:
$$
\F_1(P(n)) \simeq \OO_{\PP^1}(n) \text{~~~and~~~} \F_1(I(n))\simeq \OO_{\PP^1}(-n),
$$
where $P(n)$ and $I(n)$ are the preprojective and preinjective representations of $\mathcal{K}_2$ of dimension $2n+1$.
Moreover, any vector bundle over $\PP^1$ can be realized as $\F_1(M)$ for a module $M$ of constant Jordan type.
\end{theorem}

\begin{proof} Note first that the $H$-grading on $\s$ is precisely the standard grading by total degree and $\R$ is the standard functor that assigns to a finitely generated graded module a coherent sheaf over $\PP^1$.

Recall that the preprojective module $P(n)$ is given by
\begin{align*}
    \xymatrix{\C^n \ar@/^/[r]^{\begin{bmatrix} I_n \\ 0\end{bmatrix}}
      \ar@/_/[r]_{\begin{bmatrix} 0 \\ I_n \end{bmatrix}} & \C^{n+1}.}
\end{align*}
Let $e_{1,1},\dotsc, e_{1,n}$ and $e_{2,1},\dotsc, e_{2,n+1}$ denote bases for $\C^n$ and $\C^{n+1}$ compatible with this matrix presentation of $P(n)$.  Given any $(\alpha_1,\alpha_2) \in \mathcal{V}$ we have that $T_\alpha$ is of the block form \[\left[\begin{array}{c|c} 0_{n \times n} & 0_{n\times (n+1)} \\\hline \star &
      0_{(n+1)\times (n+1)} \end{array}\right],\] where $\star$ is the $(n+1)\times n$ matrix:

\[\left[\begin{array}{ccccc}
      \alpha_1 & 0 & 0 & \dotsc & 0\\
      \alpha_2 & \alpha_1 & 0 & \dotsc & 0\\
      0 & \alpha_2 & \alpha_1 & \dotsc & 0\\
      0 & 0 & \alpha_2 & \dotsc & 0\\
      \vdots & \vdots & \vdots & \ddots & \vdots \\
      0 & 0 & 0 & \dotsc &\alpha_2
    \end{array}\right].\]
     
This matrix has rank $n$ whenever at least one of $\alpha_1,\alpha_2$ is non-zero, and its square is zero.  In particular, $P(n)$ is a module of constant Jordan type $[2]^{n}[1]^1$. We will analyze $\mathcal{F}_1(P(n))$.  Recall that this sheaf is defined to be the quotient of $\Ker \theta_{P(n)}$ by $\Ker \theta_{P(n)} \cap
\ima \theta_{P(n)}$ (in this case, $\Ker \theta_{P(n)} \cap \ima \theta_{P(n)}=\ima \theta_{P(n)}$ since $\theta_{P(n)}^2=0$).  We will work in the category $\grmodule_H(\s)$ first and then apply the functor $\R$ to yield the resulting sheaf.  To this end, we have that the map $\theta_{P(n)}: P(n) \otimes_\C \s \rightarrow P(n) \otimes_\C \s(1)$ is given by the formula \[ \theta_{P(n)}
(m\otimes f) = X_{a_1} m \otimes Y_{a_1} f + X_{a_2}m \otimes Y_{a_2}f.\]  For convenience, $x$ will denote $Y_{a_1}$ and $y$ will denote $Y_{a_2}$.  Note that
\begin{align*}
  X_{a_1} e_{1,i} &= e_{2,i}\\
  X_{a_2} e_{1,i} &= e_{2,i+1}
\end{align*}
and $X_{a_j} e_{2,i} = 0$ by definition of the module $P(n)$.

Suppose that \[ m=\sum\limits_{i=1}^n e_{1,i}\otimes f_{1,i} +
\sum\limits_{i=1}^{n+1}  e_{2,i} \otimes f_{2,i} \in P(n) \otimes_\C
\s.\]  We will identify $m$ with the element $(f_{1,1}, f_{1,2},\dotsc, f_{1,n}, f_{2,1},\dotsc,
f_{2,n+1})\in \s^{2n+1}$.  Then
\begin{align*}
  \theta_{P(n)}(m) &= e_{2,1} \otimes f_{1,1} x +
  \sum\limits_{i=2}^{n} e_{2,i} \otimes (f_{1,i}x + f_{1,i-1} y) +
  e_{2,n+1}\otimes f_{1,n} y.
\end{align*}
Therefore, $m\in \Ker \theta_{P(n)}$ if and only if $f_{1,i}=0$ for $i=1,\dotsc, n$, and $\ima \theta_{P(n)}$ is the $\s$-submodule generated by the elements $e_{2,i} \otimes x + e_{2,i+1} \otimes y$ for $i=1,\dotsc, n$.  Applying the identification above, \[ \Ker \theta_{P(n)} / \ima \theta_{P(n)} \cong \s^{n+1}/J, \] where $J$ is the submodule generated by $xe_i + ye_{i+1}$ (here, $e_i$ denotes the $i$-th standard basis vector of $\s^{n+1}$).  We claim that $\s^{n+1}/J
\cong \bigoplus_{n' \geq 0} \s(n)_{n'}$.

Consider the graded homomorphism
\begin{align*}
  \pi: \s^{n+1} & \rightarrow \s(n) \\
   (f_0,\dotsc, f_{n}) & \mapsto \sum\limits_{i=0}^{n} (-1)^i
  y^{n-i}x^i f_i.
\end{align*}
The image of $\pi$ is clearly $\bigoplus_{n' \geq 0} \s(n)_{n'}$. We will show that the kernel of this homomorphism is precisely $J$.

Let $g_j = x e_j + ye_{j+1}$ denote one of the generators of $J$. First observe that $g_j \in \Ker \pi$ since
\begin{align*}
       \pi(g_j) &= (-1)^j y^{n-j}x^j (x) + (-1)^{j+1}y^{n-j-1}x^{j+1}
  (y)\\
  &= (-1)^j(y^{n-j}x^{j+1}-y^{n-j}x^{j+1})=0.
\end{align*}
In particular, $\pi$ factors through $\s^{n+1}/J$. Denote by $\tilde{\pi}: \s^{n+1}/J \rightarrow \s(n)$ the resulting homomorphism.  We will show that $\tilde{\pi}$ is injective. Suppose that $\underline{f}=(f_0,\dotsc, f_n)+J$ is a non-zero element in $\s^{n+1}/J$. By adding elements of the form $h(x,y) g_i$ if necessary, we claim that we can find a representative of $\underline{f}$ whose first $n$, and hence all, coordinates are zero. Indeed, let $j\leq n-1$ be such that $f_j$ is the first non-zero coordinate of $(f_0, \ldots, f_n)$, and write $f_j = h(y) + x(l(x,y))$. Note that $\underline{f}=(f_0, \ldots, f_n)-l(x,y)g_j+J$ and let us denote $(f_0, \ldots, f_n)-l(x,y)g_j$ by $(f'_0, \ldots, f'_n)$. Then:
\begin{align*}
       0 = \tilde{\pi}(\underline{f}) &= \sum\limits_{i=0}^n (-1)^i
       y^{n-i} x^i f'_i \\ &= (-1)^{j_0} y^{n-j_0}x^{j_0} h(y) +
       \sum\limits_{i=j_0+1}^n (-1)^i y^{n-i}x^i f'_i\\
       &= x^{j_0}\left((-1)^{j_0}y^{n-j_0}h(y) +
         \sum\limits_{i=j_0+1}^n (-1)^i y^{n-i}x^{i-j_0}f'_i \right).
\end{align*}
Since $\s$ is an integral domain, the second factor is zero which, evaluated at $x=0$, yields $h(y)=0$. So, the first $j$ coordinates of the new representative $(f'_0, \ldots, f'_n)$ of $\underline{f}$ are now zero. Arguing this way, we can prove that $\underline{f}=0$, i.e. $\tilde{\pi}$ is injective. Therefore, \[\Ker \theta_{P(n)}/\ima \theta_{P(n)} \cong \s^{n+1}/ \Ker \pi \cong \bigoplus\limits_{n'\geq 0} \s(n)_{n'}.\]  In particular, \[ \mathcal{F}_1(P(n)) = \mathcal{R}(\Ker \theta_{P(n)}/\ima \theta_{P(n)}) \cong\OO_{\PP^1}(n) \] since $\bigoplus_{n'\geq 0} \s(n)_{n'}$ and $\s(n)$ agree in all sufficiently high degrees (see \cite{Har}).

We now turn to the preinjective module $I(n)$ given by \[\xymatrix{ \C^{n+1} \ar@/^/[r]^{\begin{bmatrix} I_n & 0 \end{bmatrix}} \ar@/_/[r]_{\begin{bmatrix} 0 & I_n \end{bmatrix}} & \C^n}. \] It is again immediate to see that $I(n)$ is of constant Jordan type $[2]^n[1]^1$. Now, let $e_{1,1},\dotsc, e_{1,n+1}$ and $e_{2,1},\dotsc, e_{2,n}$ denote a basis for $\C^{n+1}$ and $\C^n$ compatible with the matrix presentation above. Note that
\begin{align*}
       X_{a_1} e_{1,i} &= e_{2,i},\\
       X_{a_2} e_{1,i} &= e_{2,i-1},
\end{align*}
and $X_{a_j} e_{2,i} = 0$ for $i=1,\dotsc, n$ (here we take $e_{2,0}$ and $e_{2,n+1}$ be zero). Suppose that \[ m=\sum\limits_{i=1}^{n+1} e_{1,i} \otimes f_{1,i}+\sum\limits_{i=1}^n e_{2,i} \otimes f_{2,i} \in I(n) \otimes_\C \s.\]  We will identify $m$ with the element $(f_{1,1}, f_{1,2},\dotsc, f_{1,n+1}, f_{2,1},\dotsc, f_{2,n}) \in \s^{2n+1}$. Then 
\begin{align*}
\theta_{I(n)}(m)=\sum\limits_{i=1}^n e_{2,i} \otimes (xf_{1,i} + yf_{1,i+1}),
\end{align*}
and so $m\in \Ker \theta_{P(n)}$ if and only if
\begin{align*}
xf_{1,i} + yf_{1,i+1}&=0
\end{align*}
for $i=1,\dotsc, n$, and the image is the $\s$-submodule generated by the elements $e_{2,i-1}\otimes y+e_{2,i}\otimes x$ for $i=1, \ldots, n+1$. Employing the vector notation of above,
\begin{align*}
       \Ker \theta_{I(n)}/\ima \theta_{I(n)} \cong \mathcal L \oplus
       \s^n/J
\end{align*}
 where $\mathcal L\subset \s^{n+1}$ is the set of elements $(f_0,\dotsc, f_n)\in \s^{n+1}$ with $xf_i +yf_{i+1} = 0$, and $J$ is the submodule of $\s^n$ generated by $xe_1, ye_1+xe_2,\dotsc, ye_{n-1}+xe_n, ye_n$. It is now clear that the following proposition completes the proof of the first part of our theorem:
 
\begin{prop}
With all of the above notation, $\mathcal L \cong \s(-n)$ and $\s^n/J$ is of finite length.  Therefore, $\R(\mathcal L\oplus \s^n/J) \cong \OO_{\PP^{1}}(-n)$.
\end{prop}
\begin{proof}
\label{sec:with-all-above}
We first consider $\mathcal L$.  Let $p: \s(-n) \rightarrow \s^{n+1}$ denote the homomorphism of graded $\s$-modules \[ p(f)=(y^nf,-y^{n-1}xf, \dotsc, (-1)^nx^nf). \]  It is clear by definition of $\mathcal L$ that $\ima p \subset \mathcal L$. To prove the other inclusion, pick an arbitrary element $(f_0, \ldots, f_n) \in \mathcal L \subset \s^{n+1}$.  Exploiting the equations $xf_i+yf_{i+1}=0$, a quick induction argument shows that $x^iy^{n-i}$ divides $f_i$ for all $i=0, \ldots, n$. Now, writing $f_i = (-1)^iy^{n-i}x^i f_i'$ and using the defining equations $xf_i+yf_{i+1}=0$ again, we see that
\[ x((-1)^iy^{n-i}x^if_i')+y((-1)^{i+1}y^{n-i-1}x^{i+1}f'_{i+1}) = 0.\] In particular, $f_i' = f_{i+1}'$ for $i=0,\dotsc,  n-1$.  Therefore, there exists an element $f'\in \s$ for which $f_i = (-1)^i y^{n-i}x^i f'$.  Thus, $p(f') =(f_0,\dotsc, f_n)$, so $\ima p \supset \mathcal L$.  The homomorphism $p$ is clearly injective since $\s$ is an integral domain, so $p: \s(-n) \rightarrow \mathcal L$ is an isomorphism.  

As for the second statement, we simply need to show that $\s^n / J$ is of finite length.  We will do so by showing that this module is a quotient of $(\s/\langle x^n,y^n\rangle)^{n}$, which itself has finite length.  In particular, we will show that $x^ke_k$ and $y^{n-k}e_k$ are in $J$ for each $k=1,\dotsc, n$.  Write $g_i = ye_i + xe_{i+1}$ for $i=1,\dotsc, n-1$ so that $J$ is generated by $g_i$ and $xe_1, ye_n$.
         
A simple calculation shows that, for any $k=0,\dotsc, n-2$,
\begin{align*}
\sum\limits_{i=0}^{k} (-1)^iy^{k-i}x^ig_{1+i}=y^{k+1}e_1+(-1)^{k}x^{k+1}e_{k+2}.
\end{align*}
Therefore, $y^{k-1}e_1 +(-1)^k x^{k-1}e_k\in J$, and as a result,
so is $xy^{k-1}e_1 + (-1)^k x^k e_k \in J$.  But $x(y^{k-1}e_1)\in J$ as well, so $x^ke_k \in J$.  A similar argument shows that $y^{n-k}e_k \in J$.
\end{proof}
Now, recall that any vector bundle over $\PP^1$ splits into a finite direct sum of twists of $\OO_{\PP^1}$ by a classical result of Grothendieck. Furthermore, the functor $\F_1$ is easily seen to be additive for (generalized) Kronecker quivers. Finally, we conclude that for any vector bundle $\F$ over $\PP^1$, $\F\simeq \F_1(M)$ for some module $M$ of constant Jordan type. 
\end{proof}
     
\section{Special classes of modules of constant Jordan type}\label{eip-ekp:sec} In this section, we introduce the so-called modules with the (relative) constant images/kernels properties and those with constant $j$-rank. These classes of modules were first defined in the context of modular representation theory for elementary abelian $p$-groups (see \cite{CarFriSus}, \cite{FriPev}) and for generalized Beilinson bound quiver algebras (see \cite{Wor}, \cite{Far}).

Recall from previous section that:
\begin{itemize}
\item $\sigma_0 \in \ZZ^{Q_0}$ is our fixed weight such that $\V=\A^{Q_1} \setminus \mathbb{V}(Y_{\rr} \mid \rr \in I) \neq \emptyset$ where
$$I=\{\rr=(r_a)_{a \in Q_1} \in \ZZ_{\geq 0}^{Q_1} \mid \sum_{\buildrel {a \in Q_1} \over {ta=i}} r_a-\sum_{\buildrel {a \in Q_1} \over {ha=i}} r_a=\sigma_0(i), \forall i \in Q_0 \},$$
and, for each $\rr \in I$, $Y_{\rr}=\prod_{a \in Q_1} Y_a^{r_a}$;

\item  for a module $M \in \module(KQ)$ and $\alpha \in \A^{Q_1}$, we have the nilpotent linear operator $\alpha_M:M \to M$ defined by the formula:
$$
\begin{aligned}
\alpha_M(m)=& T_{\alpha}m\\
=&\sum_{a \in Q_1}\left(\sum_{\rr \in I} r_aY_{\rr}(\alpha) \right) X_a m, \forall m \in M.
\end{aligned}
$$
\end{itemize}

Throughout this section, $0 \notin V$ is a fixed non-empty open subset of $\A^{Q_1}$. 

\subsection{The constant images/kernels properties} We say that a module $M \in \module(KQ)$ is of \emph{constant Jordan type $[L]^{a_L}\dots [2]^{a_2}[1]^{a_1}$ relative to $V$} if, for any $\alpha \in V$, the Jordan canonical form of $\alpha_M$ has $a_L$ Jordan blocks of length $L$, $\ldots$, $a_1$ Jordan blocks of length $1$. For such a module $M$, we define
$$
\jt_V(M):=[L]^{a_L}\dots [2]^{a_2}[1]^{a_1}.
$$
We denote by $\cjt_V(Q)$ the full subcategory of $\module(KQ)$ whose objects are the modules of constant Jordan type relative to $V$. When $V=\V$, we simply write $\cjt(Q)$ for $\cjt_V(Q)$.

For each non-negative integer $l$, denote by $\Gamma^l_{in}$ (resp.
$\Gamma^l_{out}$) the set of vertices $y\in Q_0$ such that there is a
path $p$ of length $l$ ending at $y$ (resp. starting at $y$).  Given a module $M \in \module(KQ)$ and an integer $l \geq 1$, we define
$$
R^l(M)=\bigoplus_{y \in \Gamma^l_{in}} M_y \text{~and~} S^l(M)=\bigoplus_{x \notin \Gamma^l_{out}} M_x.
$$

\begin{remark} Note that if $M$ is a non-simple indecomposable representation of the generalized Kronecker quiver $\mathcal K_n$ then $R^1(M)=\rad(M)$ and $S^1(M)=\soc(M)$. 
\end{remark}

For each $1 \leq l \leq L$, we are interested in the following two classes of modules:
$$\eip^l_V(Q):=\{M \in \module(KQ) \mid \ima \alpha_M^l =R^l(M), \forall \alpha \in V\}$$
and
$$\ekp^l_V(Q):=\{M \in \module(KQ) \mid \Ker \alpha_M^l=S^l(M), \forall \alpha \in V\}.$$

Of course, $\eip^L_V(Q)=\ekp^L_V(Q)=\module(KQ)$. We furthermore define:
$$
\eip_V(Q):=\bigcap_{i=1}^L \eip^l_V(Q) \text{~and~} \ekp_V(Q):=\bigcap_{i=1}^L \ekp^l_V(Q).$$

We call $\eip_V(Q)$ the class of $KQ$-modules with the \emph{equal images property} relative to $V$ and $\ekp_V(Q)$ the class of $KQ$-modules with the \emph{equal kernels property} relative to $V$. Note that:
$$
\eip_V(Q)\cup \ekp_V(Q)\subseteq \cjt_V(Q).
$$ 

\begin{example} (1) For the generalized Kronecker quiver $\mathcal K_n$, it is immediate to check (see also \cite[Section 5]{Far}) that for a module $M \in \rep(\mathcal K_n)$:
\begin{itemize}
\item $M \in \eip(\mathcal K_n)$ if and only if $\sum_{i=1}^n \alpha_i M(a_i)$ is surjective for all $\alpha \in \V=\A^n\setminus \{0\}$;
\item $M \in \ekp(\mathcal K_n)$ if and only if $\sum_{i=1}^n \alpha_i M(a_i)$ is injective for all $\alpha \in \V=\A^n\setminus \{0\}$.
\end{itemize}
\vspace{5pt}
\noindent
(2) The indecomposable modules in $\eip(\mathcal K_2)$ ($\ekp(\mathcal K_2)$) are precisely the pre-injective (pre-projective) modules of $\mathcal K_2$. Indeed, simply recall that for a pre-injective or pre-projective module $M \in \rep(\mathcal K_2)$ of dimension $2n+1$, the rank of $\sum_{i=1}^n \alpha_i M(a_i)$ is $n$ for all $\alpha \in \A^n\setminus \{0\}$.

\vspace{5pt}
\noindent
(3) When $n \geq 3$, it was proved in \cite{Wor} that for any regular component $\mathcal C$, $\mathcal C \cap \eip(\mathcal K_n)$ and $\mathcal C \cap \ekp(\mathcal K_n)$ are disjoint cones, and that the width between two such cones can be arbitrarly large.
\end{example}

We have the following simple lemma: 

\begin{lemma} (a) If $M \in \eip_V(Q)$ has constant Jordan type $[L]^{a_L}\ldots [1]^{a_1}$ then $M/R^l(M)$ belongs to $\eip_V(Q)$ and
$$
\jt_V(M/R^l(M))=[l]^{a_l+\ldots+a_L} [l-1]^{a_{l-1}}\ldots [1]^{a_1}, \forall l \geq 1.$$

(b) If $M \in \ekp_V(Q)$ has constant Jordan type $[L]^{a_L}\ldots [1]^{a_1}$ then $M/S^l(M)$ belongs to $\ekp_V(Q)$ and 
$$
\jt_V(M/S^l(M))=[L-l]^{a_L}\ldots [1]^{a_{l+1}}, \forall l \geq 1.$$ 
\end{lemma}

\begin{proof} Let $M$ be a module of constant Jordan type relative to $V$ with $\jt(M)=[L]^{a_L}\ldots [1]^{a_1}$. This is equivalent to saying that
$$
\rk(\alpha_M^{j-1})-\rk(\alpha_M^j)=a_j+\ldots +a_L, \forall \alpha \in V, \forall j \geq 1,
$$
and hence
$$
a_j=\rk(\alpha_M^{j-1})+\rk(\alpha_M^{j+1})-2\rk(\alpha_M^j), \forall \alpha \in V, \forall j \geq 1.
$$

(a) Assume that $M \in \eip_V(Q)$. Let $l \geq 1$ and denote $M/R^l(M)$ by $M'$. It is easy to see that for any $\alpha \in V$ and $j \geq 1$, $\ima \alpha^j_{M'}=(R^j(M)+R^l(M))/R^l(M)$ and so $M' \in \eip_V(Q)$. Moreover, $\rk(\alpha^j_{M'})=0$ for $j \geq l$ and $\rk(\alpha^j_{M'})=\rk(\alpha^j_M)-\rk(\alpha^l_M)$ for $j<l$. Consequently, we get  that
$$
\rk(\alpha_{M'}^{j-1})+\rk(\alpha_{M'}^{j+1})-2\rk(\alpha_{M'}^j)=
\begin{cases}
a_j&\text{~if~} j \leq l-1,\\
a_j+\ldots+ a_L&\text{~if~} j=l,\\
0&\text{otherwise},
\end{cases}
$$
and hence we get the desired formula for the Jordan type of $M'$.

(b) Assume this time that $M \in \ekp_V(Q)$. Let $l \geq 1$ and denote $M/S^l(M)$ by $M''$. It is immediate to check that for any $\alpha \in V$ and $j \geq 1$, $\Ker(\alpha^i_{M''})=S^{l+i}(M)/S^l(M)$. Using this, one immediately derives the desired formula for the Jordan type of $M''$.
\end{proof}

We show next that, after shrinking $V$ a bit, the class of modules with the constant images (kernels) property arise as the torsion-free (torsion) part of a torsion pair in $\module(KQ)$. 

Given $\alpha\in \A^{Q_1}$, for an arrow $a\in Q_1$, we write 
$$\varphi_a(\alpha)=\sum_{\mathbf{r}\in I} r_a Y_{\rr}(\alpha)$$ 
for the coefficient of $X_a$ in $T_\alpha$, and for a path $p=a_l a_{l-1} \dotsc a_1$, write 
$$\varphi_p(\alpha) = \prod_{i=1}^l \varphi_a(\alpha).$$  
(Whenever $\alpha$ is understood from the context, we simply write $\varphi_a$ and $\varphi_p$ instead of $\varphi_a(\alpha)$ and $\varphi_p(\alpha)$.)

Let $\Gamma^l(x,y)$ be the set of paths $p$ of length $l$ with $t(p)=x$ and $h(p)=y$. Consider the map
$$F_\alpha^l: \bigoplus \limits_{y\in \Gamma^l_{in}} P(y) \rightarrow \bigoplus\limits_{x\in \Gamma^l_{out}}P(x)$$ 
defined by $(F_\alpha^l)_{x,y} = \sum\limits_{p\in \Gamma^l(x,y)} \varphi_p p$.  (Here $(F_\alpha^l)_{x,y}$ is the $x$ component of the restriction of $F_\alpha^l$ to the summand $P(y)$.)  Finally, denote by $X^l_\alpha$ the cokernel of the map $F_\alpha^l$.

If we write $\widetilde{F_\alpha^l}$ for the map $\bigoplus\limits_{y\in Q_0} P(y) \rightarrow \bigoplus\limits_{x\in Q_0} P(x)$ obtained by extending $F_\alpha^l$ by zero wherever it is not defined, it is immediately clear that:
$$\Hom_{K Q}(\widetilde{F_\alpha^l}, M) = \alpha_M^l.$$ 
In particular, $\rank \alpha_M^l = \rank \Hom_{K Q}(F_\alpha^l, M)$ by definition of $\widetilde{F_\alpha^l}$.  

If $F_\alpha^l$ is injective, then the sequence 
\begin{equation}\label{eq:proj-res} 
\xymatrix@C=8ex{ 0 \ar[r] & \bigoplus\limits_{y\in \Gamma^l_{in}} P(y) \ar[r]^{F_\alpha^l} & \bigoplus\limits_{x\in \Gamma^l_{out}} P(x) \ar[r]^\pi & X^l_\alpha \ar[r] & 0} 
\end{equation} is a projective resolution of $X^l_\alpha$, which will allow for homological interpretations of $\eip_V(Q)$ and $\ekp_V(Q)$. Define:
$$
F_{inj}=\{\alpha \in \A^{Q_1} \mid F^l_{\alpha} \text{~~~is injective for~~~}l=1,\dotsc, L\}.
$$
Moreover, for an open subset $\emptyset \neq V \subseteq \A^{Q_1}$, we set 
$$V_{inj}:=V \cap F_{inj}.$$

\begin{lemma} \label{lem:ppaths} For $\alpha \in \A^{Q_1}$, $\alpha \in F_{inj}$ if and only if for each $y\in \Gamma^l_{in}$, there is at least one path $p$ of length $l$ with $h(p)=y$ and $\varphi_p(\alpha)\neq 0$. In particular, $F_{inj}$ is a non-empty open subset of $\A^{Q_1}$. 
\end{lemma}

\begin{proof} \label{sec:alpha-in-v}
Recall that the map $F_\alpha^l$ sends $v$ to $\sum\limits_{\lvert p \rvert = l} \varphi_p vp$.  Now an element $v$ in $\bigoplus_{y\in \Gamma^l_{in}} P(y)$ is of the form $v=\sum_{y\in \Gamma^l_{in}} \sum_{q: t(q)=y} \chi_q q$, so 
$$F_\alpha^l(v)=\sum\limits_{y\in \Gamma^l_{in}}~~~\sum\limits_{q:t(q)=y}~~~ \sum\limits_{\buildrel {p: h(p)=y} \over {|p|=l}} \varphi_p \chi_q qp.$$  
Since paths are linearly independent, this element is equal to zero if and only if $\varphi_p\chi_q=0$ for all concatenable paths $qp$ with $h(p)=t(q)$ and $|p|=l$. In other words, $F^l_{\alpha}$ is injective if and only if for every vertex $y\in \Gamma^l_{in}$ there exists a path $p$ with $h(p)=y$, $|p|=l$ and $\varphi_p \neq 0$. 
\end{proof}

\begin{example} (1) For the generalized Kronecker quiver $\mathcal K_n$, we clearly have that $\V_{inj}=\V$.

(2) On the other hand, in Example \ref{ex:basicexamples}{(2)}, the coefficients $\varphi_a(\alpha)$ are given below:
\begin{align*}
\varphi_{a_1}(\alpha) &= \alpha_1 \alpha_2+ \alpha_1 \alpha_3\\
\varphi_{a_2}(\alpha) &= \alpha_1 \alpha_2 + 2\alpha_2^2 \alpha_4 + \alpha_2\alpha_3\alpha_4\\
\varphi_{a_3}(\alpha) &= \alpha_1 \alpha_3 + \alpha_2\alpha_3\alpha_4 + 2\alpha_3^2 \alpha_4\\
\varphi_{a_4}(\alpha) &=\alpha_2^2 \alpha_4 + \alpha_2\alpha_3\alpha_4 + \alpha_3^2\alpha_4,
\end{align*}
and so $F_{inj}=\{\alpha \in \A^{Q_1} \mid \varphi_{a_4}(\alpha)\neq 0 \text{~and either~} \varphi_{a_2}(\alpha) \neq 0 \text{~or~} \varphi_{a_3}(\alpha)\neq 0\}$. Hence, $\V_{inj}$ is properly contained in $\V$ since $\alpha=(1,2,0, 0) \in \V\setminus \V_{inj}$.
\end{example}

\begin{question} Is there a weight $\sigma_0$ such that $\V_{inj}=\V$?
\end{question}

We have the following useful description of $\eip$ and $\ekp$ after shrinking $V$ to $V_{inj}$:

\begin{prop}
\begin{enumerate}
\renewcommand{\theenumi}{\arabic{enumi}}
\item $\eip_{V_{inj}}(Q)=\{M \in \module(KQ) \mid \Ext_{KQ}^1(X^l_{\alpha}, M)=0, \forall 1 \leq l \leq L, \forall \alpha \in V_{inj}\}$. Consequently, $\eip_{V_{inj}}$ contains all injective $KQ$-modules and is the torsion class $\mathcal T$ of a torsion pair $(\mathcal T, \mathcal F)$ in $\module(KQ)$.
\item $\ekp_{V_{inj}}(Q)=\{M \in \module(KQ) \mid \Hom_{KQ}(X^l_{\alpha}, M)=0, \forall 1 \leq l \leq L, \forall \alpha \in V_{inj}\}$. Consequently, $\ekp_{V_{inj}}(Q)$ is the torsion-free class $\mathcal F'$ of a torsion pair $(\mathcal T', \mathcal F')$ in $\module(KQ)$. 
\end{enumerate}
\end{prop}

\begin{proof} For every $M \in \module(KQ)$, $\alpha \in F_{inj}$, and $l \geq 1$, we have the long exact sequence:
$$
0 \to \Hom_Q(X^l_{\alpha},M)\to \bigoplus_{x \in \Gamma^l_{out}} M_x \to \bigoplus_{y \in \Gamma^l_{in}} M_y \to \Ext^1_Q(X^l_{\alpha}, M) \to 0,
$$
where the map in the middle is $\pi \circ \alpha_M^l\circ \tau$  with $\tau$ the canonical embedding of $\bigoplus_{x \in \Gamma^l_{out}} M_x$ into $\bigoplus_{x \in Q_0} M_x$ and $\pi$ the canonical projection of $\bigoplus_{x \in Q_0} M_x$ onto $\bigoplus_{y \in \Gamma^l_{in}} M_y$. Note also that
$$
\ima \alpha^l_M=\ima (\pi \circ \alpha_M^l\circ \tau)  \text{~~and~~} \Ker \alpha^l_M=\Ker (\pi \circ \alpha_M^l\circ \tau) \oplus \bigoplus_{x \notin \Gamma^l_{out}}M_x.
$$
Consequently, we get that for every $l \geq 1$,
$$\eip^l_{V_{inj}}(Q)=\{M \in \module(KQ) \mid \Ext_Q^1(X_{\alpha}^l,M)=0, \forall \alpha \in V_{inj}\},$$
and
$$\ekp^l_{V_{inj}}(Q)=\{M \in \module(KQ) \mid \Hom_Q(X_{\alpha}^l,M)=0, \forall \alpha \in V_{inj}\}.$$
In particular, this gives us the description of $\eip_{V_{inj}}(Q)$ and $\ekp_{V_{inj}}(Q)$ in terms of the vanishing of $\Ext$ and $\Hom$; it also shows that all injective $KQ$-modules belong to $\eip_{V_{inj}}(Q)$.

Next, let $0\rightarrow M' \rightarrow M \rightarrow M''\rightarrow 0$ be a short exact sequence in $\module(KQ)$ (in particular $M'$ is a submodule of $M$ and $M''$ is a quotient).  By applying $\Hom_{K Q}(X_\alpha^l, -)$ it is easily seen that if $M$ is in $\eip_{V_{inj}}(Q)$, then so is $M''$, while if $M$ is in $\ekp_{V_{inj}}(Q)$, so is $M'$.  It also shows that if $M', M''$ are in either class, then so is $M$. Therefore, $\eip_{V_{inj}}(Q)$ is closed under quotients, $\ekp_{V_{inj}}(Q)$ is closed under submodules, while they are both closed under extension. 

So, $\eip_{V_{inj}}(Q)$ is indeed the torsion class of a torsion pair $(\mathcal{T}, \mathcal{F})$, and $\ekp_{V_{inj}}(Q)$ is the torsion-free class of a torsion pair $(\mathcal{T}', \mathcal{F}')$.
\end{proof}

\begin{remark} It is easy to see that for any $l \geq 1$:
$$\eip^l_V(Q)\subseteq \{M \in \module(KQ) \mid \Ext_Q^1(X_{\alpha}^l,M)=0, \forall \alpha \in V\},$$
and
$$\{M \in \module(KQ) \mid \Hom_Q(X_{\alpha}^l,M)=0, \forall \alpha \in V\} \subseteq \ekp^l_V(Q).$$
\end{remark}

\begin{remark} Before we move on, we would like to point out some key differences between our non-commutative quiver set-up and Worch's set-up. In \cite{Wor}, the definition of modules of constant Jordan type and of those with the constant images/kernels properties involves \emph{commutative} elements of generalized Beilinson algebras. In particular, this commutativity allows one to immediately conclude that the linear operators $\alpha_{X^l_{\alpha}}$ are zero in the set-up of \cite{Wor}. This, in turn, is key when it comes to checking that the analogs of $\eip$ and $\ekp$ for generalized Beilinson algebras are closed under $\tau$ and $\tau^{-}$, respectively.

In our set-up, $\alpha_{X^l_{\alpha}}$ is \emph{not} necessarily zero for generic $\alpha \in \A^{Q_1}$ and this leads to obstacles in the way of checking whether $\eip_{V_{inj}}(Q)$ or $\ekp_{V_{inj}}(Q)$ is closed under Auslander-Reiten translation for arbitrary acyclic quivers $Q$. Nonetheless, we plan to address these issues in a forthcoming paper on the subject.
\end{remark}

\subsection{Modules of relative constant rank} Modules of constant rank for group schemes have been introduced by Friedlander and Pevtsova in \cite{FriPev}. We extend their definition to our quiver set-up as follows. 

\begin{definition} A module $M \in \module(KQ)$ is said to be of \emph{constant $l$-rank relative to $V$}, where $l\geq 1$, if there exists an integer $r_l$ such that $\rk(\alpha_M^l)=r_l$ for all $\alpha \in V$. 

We denote the class of all such modules by $\CR^l_V(Q)$. When $V=\V$, we simply write $\CR^l(Q)$ instead of $\CR^l_{\V}(Q)$.
\end{definition}

Obviously, one has that: 
$$\cjt_V(Q)=\bigcap_{l=1}^L \CR_V^l(Q).$$ 
Note that Proposition \ref{fiber-prop} can be rephrased as saying that:

\begin{prop} Assume that $V=\pi^{-1}(Y)$ where $Y$ is an open subvariety of $\M$ and $\pi:\V \to \M$ is the quotient morphism. Then, for a module $M \in \module(KQ)$, the restriction of $\ima \theta_M^l$ to $Y$ is a locally free sheaf on $Y$ if and only if $M \in \CR^l_V(Q)$.
\end{prop}

By analogy with the work of Worch \cite{Wor}, it is possible to give a homological interpretation of the set of modules $M$ of relative constant $l$-rank. 

\begin{prop}\label{prop:rel-const-rank}
A module $M\in \module(KQ)$ is of constant $l$-rank relative to $V_{inj}$ if and only if there exists an integer $c_l$ such that $\dim_K\Ext^1_{K Q}(X^l_\alpha, M)=c_l$ for all $\alpha \in V_{inj}$. In other words,
$$
\CR^l_{V_{inj}}(Q)=\{M \in \module(KQ) \mid \exists c_l \text{~such that~}\dim_K \Ext^1_{KQ}(X^l_{\alpha}, M)=c_l, \forall \alpha \in V_{inj}\}.
$$
\end{prop}

\begin{proof}
For $\alpha\in V_{inj}$, we have have the short exact sequence $$\xymatrix@C=2ex{0 \ar[r] & \Hom_{K Q}(X^l_\alpha, M) \ar[r] & \bigoplus\limits_{x\in \Gamma^l_{out}} \Hom_{K Q}(P(x), M) \ar[r]^\delta & \bigoplus\limits_{y\in \Gamma^l_{in}} \Hom_{KQ}(P(y), M) \ar[r] & \Ext^1_{KQ}(X_\alpha^l, M)\ar[r] & 0}.$$
We have seen that $\rank \delta = \rank \alpha_M^l$ and so, by exactness,
\begin{align*} \dim_K \Ext^1_{KQ}(X_\alpha^l, M) &= \dim_K \left(\bigoplus\limits_{y\in \Gamma^l_{in}} \Hom_{KQ}(P(y), M)\right) - \rank \delta\\ &=\left(\sum\limits_{y\in \Gamma^l_{in}} \dim_K   M_y\right) - \rank \alpha_M^l.\end{align*}  Therefore, $\rank \alpha_M^l$ is constant for all $\alpha\in V_{inj}$ if and only if $\dim \Ext^1_{KQ}(X_\alpha^l, M)$ is constant for all $\alpha \in V_{inj}$.  
\end{proof}

We end this section with a result on modules of constant rank and maps to Grassmannians. This is the quiver analog of Farnsteiner's Theorems 5.4.1 and 5.4.3 from \cite{Far} (see also \cite{Far2}):

\begin{prop} Assume that $V=\pi^{-1}(Y)$ where $Y$ is an open subvariety of $\M$ and $\pi:\V \to \M$ is the quotient morphism. If $Y=\mathbb P^n$ and $M \in \CR^l_V(Q)$ with $\dim_K R^l(M) \leq n$ then $\ima \alpha_M^l$ does not change for $\alpha \in V$. 

In particular, if $M \in \cjt(\mathcal K_{n+1})$, $\dim_K M_1\leq n$, and $S(1)$ is not a direct summand of $M$ then $M \in \eip(\mathcal K_{n+1})$.
\end{prop}

\begin{proof} Let $r_l$ be the common rank of the linear operators $\alpha_M^l$, $\alpha \in V$. The map $V \to \Gr_{r_l}(R^l(M))$, sending $\alpha \in V$ to $\ima \alpha_M^l$, is a $T$-invariant morphism. (Recall that for any $\beta \in \A^{Q_1}$, $l \geq 1$, and $N \in \module(KQ)$, we have that $\ima \beta_M^l \subseteq R^l(N)$.) Furthermore, $Y$ is a good quotient of $V$ by $T$ with quotient morphism the restriction of $\pi$ to $V$. Hence, it follows from the universal property of quotient varieties that there exists a unique morphism 
$$
\ima_M^l: Y \to \Gr_{r_l}(R^l(M))
$$
such that any $\overline{\alpha}=\pi(\alpha) \in Y$, with $\alpha \in V$, is mapped to $\ima \alpha_M^l$. Consequently, we have that:
$$
\ima_M^l \text{~is constant~} \Longleftrightarrow \ima \alpha_M^l \text{~does not change for~} \alpha \in V.
$$
The first claim of the proposition now follows from Tango's results on morphisms from projective spaces to Grassmannians (see \cite{Tan}).

When $Q=\mathcal K_{n+1}$, we have that $\V=\A^{n+1}\setminus \{0\}$, $\M=\mathbb P^n$, and $\cjt(\mathcal K_{n+1})=\CR^1(\mathcal K_{n+1})$. Consequently, for $M \in \cjt(\mathcal K_{n+1})$ with $\dim_K M_1 \leq n$, we have that $R:=\ima \alpha_M$ does not change for $\alpha \in \V$. Hence, $R=\rad(M)=M_1$, i.e. $M \in \eip(\mathcal K_{n+1})$.
\end{proof}

\section{The exact category $\cjt(Q)$}\label{exact-cjt:sec} In modular representation theory of finite group schemes, Carlson and Friedlander in \cite{CarFri} equipped the additive category of modules of constant Jordan type with an exact structure in the sense of Quillen. In this section, we adapt Carlson-Friedlander's approach to our quiver set-up. 

Let $V$ be a non-empty open subset of $\A^{Q_1}$. Given a module $M \in \module(KQ)$ and an integer $i \geq 1$, consider the map $f_i:\A^{Q_1} \to \ZZ_{\geq 0}$, $f_i(\alpha)=\rk \alpha_M^i$, which is easily seen to be lower semi-continuous. Denote by $\Max_i(M)=\Max \{f_i(\alpha) \mid \alpha \in \A^{Q_1}\}$ and define
$$
V_{\Max}(M)=\{\alpha \in V \mid \rk \alpha_M^i= \Max_i(M), \forall i \geq 1\},
$$
which is a non-empty open subset of $V$. We have the following simple facts:
\begin{itemize}
\item $M \in \cjt_V(Q)$ if and only if $V_{\Max}(M)=V$;
\item $V_{\Max}(M_1 \oplus M_2)=V_{\Max}(M_1)\cap V_{\Max}(M_2)$.
\end{itemize}
As an immediate consequence, we have:

\begin{lemma} The category $\cjt_V(Q)$ is closed under isomorphisms, direct sums, and direct summands. 
\end{lemma}

Recall that if $M \in \module(KQ)$ and $\alpha \in \A^{Q_1}$, $\alpha^*(M)$ denotes the pull-back of $M$ along the algebra homomorphism $K[t]/(t^L)\to KQ$ defined by sending $t+(t^L)$ to $T_{\alpha}$. Note that $M \in \cjt_V(Q)$ if and only if the decomposition of $\alpha^*(M)$ into indecomposable $K[t]/(t^l)$-modules does not depend on the choice of $\alpha$ in $V$.

Following Carlson and Friedlander in \cite{CarFri}, we say that a short exact sequence
$$
0 \to M_1 \to M_2 \to M_3\to 0
$$
in $\module(KQ)$ is \emph{locally split} if, for every $\alpha \in V$, the exact sequence $$0 \to \alpha^*(M_1) \to \alpha^*(M_2) \to \alpha^*(M_3) \to 0$$ splits in $\module(K[t]/(t^L))$. A locally split sequence is also refereed to as an \emph{admissible sequence}. Furthermore, a homomorphism appearing as the first map in an admissible sequence is called an \emph{admissible monomorphism}. Similarly, a homomorphism appearing as the second map of an admissible sequence is called an \emph{admissible epimorphism}.

If $0 \to M_1 \to M_2 \to M_3\to 0$ is a locally split exact sequence then it is immediate to see that $\rk(\alpha_{M_2}^i)=\rk(\alpha_{M_1}^i)+\rk(\alpha_{M_3}^i)$ for all $\alpha \in V$ and $i \geq 1$, and hence $V_{\Max}(M_2)=V_{\Max}(M_1)\cap V_{\Max}(M_3)$. Consequently, we have:

\begin{prop}\label{prop-split-ext} If $0 \to M_1 \to M_2 \to M_3\to 0$ is a locally split extension then $M_1, M_3 \in \cjt_V(Q)$ if and only if $M_2 \in \cjt_V(Q)$, i.e. $\cjt_V(Q)$ is closed under locally split extensions.
\end{prop}

With this proposition in mind and using the strategy from \cite{CarFri}, one can easily check that $\cjt_V(Q)$ has an exact category structure (see also \cite[Proposition 1.5]{CarFri}):

\begin{prop} The category $\cjt_V(Q)$ together with the class of admissible sequences is an exact category in the sense of Quillen.
\end{prop}

\begin{proof} According to Keller \cite[Appendix A]{Kel} (see also \cite[Section 4.3]{Ben2}), to verify that the admissible sequences define an exact structure on $\cjt_V(Q)$, we need to check three properties. The first property consists of the conditions that any sequence isomorphic to an admissible one is admissible; if $0 \to M_1 \to M_2 \to M_3 \to 0$ is admissible then $M_1 \to M_2$ is the kernel of $M_2 \to M_3$ and $M_2 \to M_3$ is the cokernel of $M_1 \to M_2$. It is obvious that these conditions hold.

The second property consists of the conditions that the identity map $0\to0$ is an admissible epimorphism; the push-out/pull-back of any morphism and an admissible monomorphism/epimorphism exists and is an admissible monomorphism/epimorphism. We will only check the condition on the pull-back: Let $0 \to Y \to M_1 \to M_2\to 0$ be an admissible sequence with $Y,M_1$ and $M_2$ in $\cjt_V(Q)$ and let $g:X\to M_2$ be a homomorphism with $X \in \cjt_V(Q)$. We then have the commutative diagram in $\module(KQ)$:
$$
\vcenter{\hbox{  
\begin{tikzpicture}[node distance=1.5cm, auto]
\node (M_1) {$M_1$};
\node (M'_1) [above of=M_1] {$M'_1$};
\node (Y) [left of=M_1] {$Y$};
\node (M_2) [right of=M_1] {$M_2$};
\node (Y') [above of=Y] {$Y$};
\node (X) [right of=M'_1] {$X$};
\node (O1) [left of=Y'] {$0$};
\node (O2) [right of=X] {$0$};
\node (O3) [below of=O1] {$0$};
\node (O4) [below of=O2] {$0$};

\draw[->] (M_1) to node {f} (M_2);
\draw[->] (Y) to node {} (M_1);
\draw[->] (M'_1) to node {u} (M_1);
\draw[->] (Y') to node {Id} (Y);
\draw[->] (X) to node {g} (M_2);
\draw[->] (Y') to node {} (M'_1);
\draw[->] (M'_1) to node {v} (X);
\draw[->] (O1) to node {} (Y');
\draw[->] (X) to node {} (O2);
\draw[->] (O3) to node {} (Y);
\draw[->] (M_2) to node {} (O4);
\end{tikzpicture}
}}
$$
where $(M'_1,u,v)$ is the pull-back of $(M_2,f,g)$. Since for each $\alpha \in V$, $\alpha^*$ applied to the bottom row splits in $K[t]/(t^L)$, the same is true for $\alpha^*$ applied to the top row. In other words, $0 \to Y \to M'_1\to X\to0$ is an admissible sequence. Using Proposition \ref{prop-split-ext}, we conclude that $M'_1 \in \cjt_V(Q)$ and so the pull-back  of the morphisms in question exists in $\cjt_V(Q)$ and is an admissible epimorphism. The condition on the push-out is proved similarly.

Finally, the third property asserts that the composition of two admissible epimorphisms is an admissible epimorphism. Let $f:M_1\to M_2$ and $h:M_2\to M_3$ be two admissible epimorphisms. Consider the induced commutative diagram:

$$
\vcenter{\hbox{  
\begin{tikzpicture}[node distance=1.5cm, auto]
\node (M_1) {$M_1$};
\node (M'_1) [above of=M_1] {$M'_1$};
\node (Y) [left of=M_1] {$Y$};
\node (M_2) [right of=M_1] {$M_2$};
\node (Y') [above of=Y] {$Y$};
\node (X) [right of=M'_1] {$X$};
\node (M_3) [below of=M_2] {$M_3$};
\node (O1) [left of=Y'] {$0$};
\node (O2) [right of=X] {$0$};
\node (O3) [below of=O1] {$0$};
\node (O4) [below of=O2] {$0$};
\node (O5) [below of=M_3] {$0$};
\node (O6) [above of=X] {$0$};

\draw[->] (M_1) to node {f} (M_2);
\draw[->] (Y) to node {} (M_1);
\draw[->] (M'_1) to node {u} (M_1);
\draw[->] (Y') to node {Id} (Y);
\draw[->] (X) to node {g} (M_2);
\draw[->] (M_2) to node {h} (M_3);
\draw[->] (Y') to node {} (M'_1);
\draw[->] (M'_1) to node {v} (X);
\draw[->] (O1) to node {} (Y');
\draw[->] (X) to node {} (O2);
\draw[->] (O3) to node {} (Y);
\draw[->] (M_2) to node {} (O4);
\draw[->] (M_3) to node {} (O5);
\draw[->] (O6) to node {} (X);
\end{tikzpicture}
}}
$$
where the bottom row and the rightmost column are admissible sequences and $(M'_1, u, v)$ is the pull-back of $(M_2,f,g)$. It is easy to check that $0\to M'_1\to M_1\to M_3 \to 0$ is a short exact sequence in $\module(KQ)$. Moreover, for any $\alpha \in V$, we have that $\alpha^*(M_1)\simeq \alpha^*(Y)\oplus \alpha^*(M_2)\simeq \alpha^*(Y)\oplus \alpha^*(X)\oplus \alpha^*(M_3)\simeq \alpha^*(M'_1)\oplus\alpha^*(M_3)$ in $K[t]/(t^L)$. This is equivalent to saying that $0\to M'_1\to M_1\to M_3$ is admissible and hence $h \circ f:M_1 \to M_3$ is, indeed, an admissible epimorphism.
\end{proof}

Now, we can define the Grothendieck group $K_0(\cjt_V(Q))$ of $\cjt_V(Q)$ to be the quotient of the free abelian group whose generators are the symbols $[M]$ corresponding to the isomorphism classes of modules $M \in \cjt_V(Q)$ modulo the subgroup generated by elements of the form $[M_1]-[M_2]+[M_3]$ for all admissible sequences $0\to M_1 \to M_2 \to M_3 \to 0$. The elements of $\cjt_V(Q)$ are called \emph{virtual} representations of $Q$ of relative constant Jordan type.

We have a group homomorphism:
$$
\jt_V: K_0(\cjt_V(Q)) \to \ZZ^L
$$
defined by sending $[M]$, where $M$ has constant Jordan type relative to $V$, to $\jt_V(M)$. The result below is the quiver analog of Proposition 3.1 in \cite{CarFri}:

\begin{prop}\label{prop:jordantypesurjectivity}
The map $\jt_{V_{inj}}$ is surjective.
\end{prop}

\begin{proof}
  \label{sec:map-jt_v_-surj}
Denote by $I(x)$ the injective envelope of the simple $K Q$-module supported at the vertex $x$.  Recall that $I(x)_y$ has basis consisting of paths $p$ with $tp=y$ and $hp=x$.  Call $q$ an initial subpath of $p$ if there is a path $r$ with $rq=p$.  Denote such a path $r$ by $p\setminus q$.  Then the action of a path $q\in KQ$ on an element $p\in I(x)$ is given by 
$$q\cdot p =
\begin{cases}
    p\setminus q & \textrm{~if~} q \textrm{~is an initial subpath of~}
 p\\
 0 & \textrm{~otherwise~}
\end{cases}.$$  
For any $\alpha\in \mathbb{A}^{Q_1}$, we view $\alpha_{I(x)}^l$ as a matrix in the basis exhibited above. Note that its $(p,q)$-entry is: 
$$
(\alpha_{I(x)}^l)_{p,q}=
\begin{cases}
  \varphi_\alpha(r) & \textrm{~if there is a path~} r \textrm{~of
    length~} l \textrm{~with~}
  rp=q\\
  0 & \textrm{~otherwise~}
\end{cases}.$$  (In the first case we say that $p$ can be preceded by the path $r$.)  Thus, the only rows in which there could be non-zero entries are those rows corresponding to paths terminating at the vertex $x$ which can be preceded by a path of length $l$. Furthermore, each column contains at most one non-zero entry.

Now suppose that $\alpha \in F_{inj} \cap V$, which is non-empty since both $F_{inj}$ and $V$ are non-empty open sets in an irreducible variety.  For $l=0,\dotsc, L-1$, denote by $Q_0(l)$ the set of vertices $x$ for which the longest path in $Q$ terminating at $x$ is of length $l$.  For any $x\in Q_0(l)$, we claim that $\alpha_{I(x)}^l$ has rank precisely one.  To prove this claim, first notice that the only row of the matrix of $\alpha_{I(x)}^l$ in which there could be a non-zero entry is a row corresponding to a path ending at $x$ that can be preceded by a path of length $l$.  The only such path is the trivial path $e_x$, so the rank of $\alpha_{I(x)}^l$ is at most one. Now by Lemma {~\ref{lem:ppaths}} there is a path $r$ of length $l$ terminating at $x$ such that $\varphi_\alpha(r)\neq 0$, thus $(\alpha_{I(x)}^l)_{e_x,r}=\varphi_\alpha(r)\neq 0$, so the rank is at least one.  

Finally, $\alpha_{I(x)}^{l+j}=0$ for $j>0$ since there are no paths in $Q$ of length more than $l$ terminating at $x$.  In particular, there is precisely one $(l+1)\times (l+1)$ Jordan block in the Jordan form of $\alpha_{I(x)}$, so $\jt_{V_{inj}}(I(x)) = E_{l+1} +\sum_{j\leq l} \gamma_j E_j$ where $E_j$ denotes the $j$-th standard basis vector in $\mathbb{Z}^L$.  Taking a collection of vertices $x_l \in Q_0(l)$ for $l=0,\dotsc, L-1$, the collection $\{\jt_{V_{inj}}(I(x_l)) \mid l=0, \dotsc, L-1\}$ is a $\mathbb{Z}$-basis of $\mathbb{Z}^L$.  In particular, $\jt_{V_{inj}}$ is surjective.
\end{proof}

\begin{remark} Proposition \ref{prop:jordantypesurjectivity} simply says that any vector of $\ZZ^L$ can be realized as the Jordan type of a virtual $KQ$-module of constant Jordan type relative to $V_{inj}$. This gives a partial answer to the very difficult \emph{algebraic realization problem} for modules of constant Jordan type which asks to describe those $L$-tuples $(a_L, \ldots, a_1) \in \ZZ^L_{\geq 0}$ that can be realized as Jordan types of modules of constant Jordan type. We plan to address this problem in a forthcoming paper on the subject. Finally, we mention that the realization problem is related, via Theorem \ref{cjt-bundles-thm}, to the notoriously difficult problem of finding indecomposable vector bundles of small rank over projective spaces. It is our hope that the quiver representation theoretic approach described above will shed light on the construction of such vector bundles. 
\end{remark}

\end{document}